\documentclass{amsart}

\usepackage{geometry} % this sets the "usual" text width and margins
\usepackage{amssymb, mathrsfs}
\usepackage{relsize}
\usepackage[all]{xy}
\usepackage{xcolor}
\usepackage{dsfont}
\usepackage[hidelinks]{hyperref}
\usepackage{cleveref}

\numberwithin{equation}{section}

\newtheorem{theorem}[subsubsection]{Theorem}
\newtheorem{proposition}[subsubsection]{Proposition}

\newtheorem{corollary}[subsubsection]{Corollary}
\newtheorem{lemma}[subsubsection]{Lemma}

\newtheorem*{theorem*}{Theorem}

\usepackage{mathtools}

\theoremstyle{definition}
\newtheorem{definition}[subsubsection]{Definition}

\theoremstyle{remark}
\newtheorem{example}[subsubsection]{Example}
\newtheorem{remark}[subsubsection]{Remark}

\title[]{Another counterexample to the Nerves of Steel Conjecture}
\author{Kevin Coulembier}
\address{School of Mathematics and Statistics, University of Sydney, Australia}
\email{kevin.coulembier@sydney.edu.au}

\newcommand{\bk}{\Bbbk}

\newcommand{\cA}{\mathcal{A}}
\newcommand{\cB}{\mathcal{B}}
\newcommand{\cL}{\mathcal{L}}

\newcommand{\Spc}{\operatorname{Spc}}

\newcommand{\unit}{\mathds{1}}

\newcommand{\Rep}{\mathtt{Rep}}

\newcommand{\GL}{\operatorname{GL}}

\newcommand{\fp}{\mathrm{fp}}

\newcommand{\univ}{\mathrm{univ}}

\newcommand{\mZ}{\mathbb{Z}}
\newcommand{\mN}{\mathbb{N}}

\newcommand{\cI}{\mathcal{I}}

\newcommand{\cK}{\mathcal{K}}
\newcommand{\cC}{\mathcal{C}}

\newcommand{\cT}{\mathcal{T}}

\newcommand{\id}{\mathrm{id}}
\newcommand{\ev}{\mathrm{ev}}
\newcommand{\coev}{\mathrm{coev}}

\newcommand{\Hom}{\operatorname{Hom}}
\newcommand{\End}{\operatorname{End}}

\newcommand{\K}{\mathtt{K}}

\renewcommand{\part}{\mathscr{P}}

\begin{document}

\begin{abstract}
In \cite{Balmer20} Balmer introduced the homological spectrum of a rigid tensor-triangulated category, and had the ``nerves of steel'' not to conjecture that it was in bijection with the prime spectrum, despite an ``avalanche of examples''. Naturally, the statement came to be known as Balmer's Nerves of Steel Conjecture. More recently, in \cite{NVY} the conjecture was stated explicitly and extended to a more general non-symmetric context. In this paper we provide a counterexample to the original suggestion in \cite{Balmer20}, so also to the conjecture from \cite{NVY}. A first counterexample was recently constructed in \cite{BHR}.
\end{abstract}


\maketitle

%%%%%%%%%%%%%%%%%%%%%%%%%%%%%%%%%%%%%%%%%%%%%%%%%%%%%%%%%%%%%%%%%%%%%%%%%%%%%%%%%%%%%%%%%%%%%%%%%%%%%%%%%%%%%%%%%%%%%%%%%%%%%%%%%%%%%%%%%%%%%%%%%%%%%%%%%%%%%%%%%%%%%%%%%%%%%%%%%%%%%%%%%%%%%%%%%%%%%%%%%%%%%%%%%%%%%%%%%%%%%%%%%%%%%%%%%%%%%%%%%%%%%%%%%%%%%%%%%%%%%%%%%%%%%%%%%%%%%%%%%%%%%%%%%%%%%%%%%%%%%%%%%%%%%%%%%%%%%%%%%%%%%%%%%%%%%%%%%%%%%%%%%%%%%%%%%%%%%%%%%%%%%%%%%%%%%%%%%%%%%%%%%%%%%%%%%%%%%%%%%%%%%%%%%%%%%%%%%%%%%%%%%%%%%%%%%%%%%%%%%%%%%%%%%%%%%%%%%%%%%%%%%%%%%%%%%%%%%%%%%%%%%%%%%%%%%%%%%%%%%%%%%%%%%%%%%%%%%%%%%%%%%%%%%%%%%%%%%%%%%%%%%%%%%%%%%%%%%%%%%%%%%%%%%%%%%%%%%%%%%%%%%%%%%%%%%%%%%%%%%%%%%%%%%%%%%%%%%%%%%%%%%%%%%%%%%%%%%%%%%%%%%%%%%%%%%%%%%%%%%%%%%%%%%%%%%%%%%%%%%%%%%%%%%%%%%%%%%%%%%%%%%%%%%%%%%%%%%%%%%%%%%%%%%%%%%%%%%%%%%%%%%%%%%%%%%%%%%%%%%%%%%%%%%%%%%%%%%%%%%%%%%%%%%%%%%%%%%%%%%%%%%%%%%%%%%%%%%%%%%%%%%%%%%%%%%%%%%%%%%%%%%%%%%%%%%%%%%%%%%%%%%%%%%%%%%%%%%%%%%%%%%%%%%%%%%%%%%%%%%%%%%%%%%%%%%%%%%%%%%%%%%%%%%%%%%%%%%%%%%%%%%%%%%%%%%%%%%%%%%%%%%%%%%%%%%%%%%%%%%%%%%%%%%%%%%%%%%%%%%%%%%%%%%%%%%%%%%%%%%%%%%%%%%%%%%%%%%%%%%%%%%%%%%%%%%%%%%%%%%%%%%%%%%%%%%%%%%%%%%%%%%%%%%%%%%%%%%%%%%%%%%%%%%%%%%%%%%%%%%%%%%%%%%%%%%%%%%%%%%%%%%%%%%%%%%%%%%%%%%%%%%%%%%%%%%%%%%%%%%%%%%%%%%%%%%%%%%%%%%%%%%%%%%%%%%%%%%%%%%%%%%%%%%%%%%%%%%%%%%%%%%%%%%%%%%%%%%%%%%%%%%%%%%%%%%%%%%%%%%%%%%%%%%%%%%%%%%%%%%%%%%%%%%%%%%%%%%%%%%%%

\section{Introduction}

In his work on nilpotence theorems for tensor-triangulated categories via homological residue fields in \cite{Balmer20}, Balmer introduced the homological spectrum of a tensor-triangulated category. It comes with a comparison map to the Balmer spectrum \cite{Balmer05} of the same category. In the rigid case it was shown that the comparison map is always surjective. Moreover, it was demonstrated in \cite[\S 5]{Balmer20} that the comparison map is even a bijection in many important examples such as the derived category of perfect complexes over schemes and categories of spectra.

The suggestion in \cite[Remark~5.15]{Balmer20} that this surjection might always be a bijection has become known as the Nerves of Steel Conjecture. Further progress, including several equivalent
formulations, was obtained in \cite{Balmer20b, BHS, H}. The homological spectrum was extended to a non-symmetric setting in \cite{NVY} where a generalisation of the conjecture was stated as \cite[Conjecture~5.6]{NVY}.

In this note we construct a counterexample to the conjecture. To explain it we can use the result in \cite[Theorem~A.1]{Balmer20b} which shows that failure of the Nerves of Steel Conjecture for a local rigid tensor-triangulated category is equivalent to the existence of two morphisms $f,g$ with $f\otimes g=0$ but for which neither morphism is $\otimes$-nilpotent (even when tensoring with an auxiliary non-zero object). Our counterexample is constructed as follows. First we consider a universal additive rigid monoidal category with two morphisms $f,g$ out of the tensor unit with $f\otimes g=0$. Subsequently we take the bounded homotopy category to obtain a rigid tensor-triangulated category. It remains to prove that the universal relations
impose no tensor-nilpotence. To achieve this we exploit that by construction our category lives above well-understood semisimple Deligne categories, in which we can demonstrate that no nilpotence arises.

\begin{remark}
At the circulation of the first draft of the paper, it was pointed out to the author that a counterexample to the conjecture was already announced 6 months earlier in \cite{BHR}. Since the construction in \cite{BHR} uses more advanced machinery (although it also uses semisimple Deligne categories) it still seemed worthwhile to make the current counterexample available. Interestingly, recently it was also proved in \cite{Hyslop2} that for rigid tensor-triangulated categories over~$\mathbb{Q}$ (such as the counterexample in this paper) satisfying Schur-finiteness (contrary to our example), the Nerves of Steel Conjecture holds true.
\end{remark}

%%%%%%%%%%%%%%%%%%%%%%%%%%%%%%%%%%%%%%%%%%%%%%%%%%%%%%%%%%%%%%%%%%%%%%%%%%%%%%%%%%%%%%%%%%%%%%%%%%%%%%%%%%%%%%%%%%%%%%%%%%%%%%%%%%%%%%%%%%%%%%%%%%%%%%%%%%%%%%%%%%%%%%%%%%%%%%%%%%%%%%%%%%%%%%%%%%%%%%%%%%%%%%%%%%%%%%%%%%%%%%%%%%%%%%%%%%%%%%%%%%%%%%%%%%%%%%%%%%%%%%%%%%%%%%%%%%%%%%%%%%%%%%%%%%%%%%%%%%%%%%%%%%%%%%%%%%%%%%%%%%%%%%%%%%%%%%%%%%%%%%%%%%%%%%%%%%%%%%%%%%%%%%%%%%%%%%%%%%%%%%%%%%%%%%%%%%%%%%%%%%%%%%%%%%%%%%%%%%%%%%%%%%%%%%%%%%%%%%%%%%%%%%%%%%%%%%%%%%%%%%%%%%%%%%%%%%%%%%%%%%%%%%%%%%%%%%%%%%%%%%%%%%%%%%%%%%%%%%%%%%%%%%%%%%%%%%%%%%%%%%%%%%%%%%%%%%%%%%%%%%%%%%%%%%%%%%%%%%%%%%%%%%%%%%%%%%%%%%%%%%%%%%%%%%%%%%%%%%%%%%%%%%%%%%%%%%%%%%%%%%%%%%%%%%%%%%%%%%%%%%%%%%%%%%%%%%%%%%%%%%%%%%%%%%%%%%%%%%%%%%%%%%%%%%%%%%%%%%%%%%%%%%%%%%%%%%%%%%%%%%%%%%%%%%%%%%%%%%%%%%%%%%%%%%%%%%%%%%%%%%%%%%%%%%%%%%%%%%%%%%%%%%%%%%%%%%%%%%%%%%%%%%%%%%%%%%%%%%%%%%%%%%%%%%%%%%%%%%%%%%%%%%%%%%%%%%%%%%%%%%%%%%%%%%%%%%%%%%%%%%%%%%%%%%%%%%%%%%%%%%%%%%%%%%%%%%%%%%%%%%%%%%%%%%%%%%%%%%%%%%%%%%%%%%%%%%%%%%%%%%%%%%%%%%%%%%%%%%%%%%%%%%%%%%%%%%%%%%%%%%%%%%%%%%%%%%%%%%%%%%%%%%%%%%%%%%%%%%%%%%%%%%%%%%%%%%%%%%%%%%%%%%%%%%%%%%%%%%%%%%%%%%%%%%%%%%%%%%%%%%%%%%%%%%%%%%%%%%%%%%%%%%%%%%%%%%%%%%%%%%%%%%%%%%%%%%%%%%%%%%%%%%%%%%%%%%%%%%%%%%%%%%%%%%%%%%%%%%%%%%%%%%%%%%%%%%%%%%%%%%%%%%%%%%%%%%%%%%%%%%%%%%%%%%%%%%%%%%%%%%%%%%%%%%%%%%%%%%%%%%%%%%%%%%%%%%%%%%%%%%%

\section{Preliminaries}

\subsubsection{Conventions} We let $\bk$ be a field. We consider $\mZ\times \mZ$ as a partially ordered set where $(m,n)\le (a,b)$ if and only if $m\le a$ and $n\le b$.

We refer to \cite{EGNO} for background on monoidal categories. We call a $\bk$-linear symmetric monoidal category $(\cC,\otimes,\unit)$ with bilinear tensor product a \emph{tensor category}, and a $\bk$-linear symmetric monoidal functor a \emph{tensor functor}.

\subsection{Three flavours of monoidal categories and ideals}

\subsubsection{} Without additional qualifiers, a \emph{tensor ideal} in a tensor category refers to an ideal in the additive sense that is closed under tensor products with arbitrary objects (equivalently morphisms). For a tensor ideal $\cI$ in a tensor category $\cC$ we can consider the quotient category $\cC/\cI$ which has the same objects but the appropriate quotients for morphism spaces. 

\subsubsection{} An essentially small tensor category $(\cC,\otimes,\unit)$  will be called a $\bk$-RSM (rigid symmetric monoidal) category if it is idempotent complete and rigid. The latter means that every object $X\in\cC$ is rigid: it has a monoidal dual $(X^\ast,\ev_X,\coev_X)$, with (co)evaluation morphisms $\ev_X:X^\ast\otimes X\to\unit$ and $\coev_X:\unit\to X\otimes X^\ast$ satisfying the usual relations. Using the braiding $\sigma$, we can then define the \emph{categorical dimension} of $X$ as
\[\dim(X)=\ev_X\circ \sigma_{X,X^\ast}\circ\coev_X\in \End(\unit).\]

It is well-known that a $\bk$-RSM category for which $\bk\to\End(\unit)$ is an isomorphism has a unique maximal tensor ideal, see \cite{EO} and references therein, and the corresponding quotient is a $\bk$-RSM category known as the \emph{semisimplification}, since it is semisimple in favourable circumstances.

\subsubsection{} For a tensor category for which the underlying additive category is abelian (but not necessarily with exact tensor product), a \emph{Serre tensor ideal} is a Serre subcategory that is closed under taking tensor products with arbitrary objects. For an abelian tensor category $\cC$ and a Serre tensor ideal $I$, we can consider the Serre quotient $\cC/\hspace{-1mm}/I$. Under suitable hypotheses, which will be satisfied in the setting relevant to the Nerves of Steel Conjecture by \cite[Proposition~2.13]{BKS}, this is again an abelian tensor category, for which the exact quotient functor $\cC\to\cC/\hspace{-1mm}/I$ becomes a tensor functor.

\subsubsection{} A \emph{tensor-triangulated category} $\cK$ is a tensor category that is simultaneously a triangulated category such that the tensor product is triangulated in each variable. A \emph{thick tensor ideal} of a tensor-triangulated category is a thick, triangulated full subcategory that is closed under taking tensor products with arbitrary objects. The prime thick tensor ideals form a topological space $\Spc\cK$, see \cite{Balmer05}, the \emph{Balmer spectrum}.

In \cite{Balmer20} the topological space $\Spc^{\mathrm{h}}\cK$ of \emph{homological primes} was introduced, as the collection of maximal Serre tensor ideals in the Freyd envelope of $\cK$, with a continuous (inverse image) map
\begin{equation}\label{eq:twospec}
\Spc^{\mathrm{h}}\cK\;\to\; \Spc\cK.
\end{equation}
By \cite[Corollary~3.9]{Balmer20}, the map in \eqref{eq:twospec} is surjective if $\cK$ is rigid.

\subsection{The Nerves of Steel Conjecture}

For a precise formulation of the conjecture, we adopt the framework from the non-symmetric generalisation in \cite{NVY}.

\begin{definition}\label{def:big}
A rigidly-compactly generated tensor-triangulated category is a tensor-triangulated category $\cL$ such that
\begin{enumerate}
\item $\cL$ admits arbitrary set-indexed coproducts;
\item the full subcategory $\cL^c$ of compact objects (those $X$ for which $\Hom(X,-)$ commutes with coproducts) is essentially small and generates $\cL$: for every $0\not=M\in\cL$ there is $X\in\cL^c$ with $\Hom(X,M)\not=0$;
\item the tensor product commutes with coproducts;
\item $\cL^c$ is a rigid monoidal subcategory.
\end{enumerate}
Note that (4) is equivalent to demanding that rigid objects and compact objects coincide.
\end{definition}
The \emph{Nerves of Steel Conjecture} then states that for every $\cK$ of the form $\cL^c$ with $\cL$ as above, the map~\eqref{eq:twospec} is a bijection, see \cite[Conjecture~5.6]{NVY}, or equivalently a homeomorphism, see \cite[Proposition~4.5]{BHS}.

\begin{example}
Let $\cA$ be a $\bk$-RSM category. Then its bounded homotopy category $\K^b(\cA)$ is a rigid tensor-triangulated category. For completeness, we show that it is of the form relevant to the above formulation of the Nerves of Steel Conjecture. 
Consider the category $\cA^\infty$ of formal set-indexed coproducts in $\cA$, for instance as a full subcategory of the $\bk$-linear presheaf category of $\cA$, which is again a tensor category by letting the tensor product commute with coproducts. Then we define $\cL$ as the smallest localising (full triangulated and closed under coproducts) subcategory of the unbounded homotopy category $\K(\cA^\infty)$ that contains $\K^b(\cA)$. Then $\cL$ satisfies the condition from Definition~\ref{def:big}, in particular $\cL^c=\K^b(\cA)$ by an application of \cite[Lemma~2.2]{N}.
\end{example}

\begin{remark}
Usually, tensor-triangulated categories need not be linear over a field $\bk$, so the Nerves of Steel Conjecture actually refers to a much broader class of categories. Since our counterexample will be $\bk$-linear, we work in this more specific framework from the start.
\end{remark}

%%%%%%%%%%%%%%%%%%%%%%%%%%%%%%%%%%%%%%%%%%%%%%%%%%%%%%%%%%%%%%%%%%%%%%%%%%%%%%%%%%%%%%%%%%%%%%%%%%%%%%%%%%%%%%%%%%%%%%%%%%%%%%%%%%%%%%%%%%%%%%%%%%%%%%%%%%%%%%%%%%%%%%%%%%%%%%%%%%%%%%%%%%%%%%%%%%%%%%%%%%%%%%%%%%%%%%%%%%%%%%%%%%%%%%%%%%%%%%%%%%%%%%%%%%%%%%%%%%%%%%%%%%%%%%%%%%%%%%%%%%%%%%%%%%%%%%%%%%%%%%%%%%%%%%%%%%%%%%%%%%%%%%%%%%%%%%%%%%%%%%%%%%%%%%%%%%%%%%%%%%%%%%%%%%%%%%%%%%%%%%%%%%%%%%%%%%%%%%%%%%%%%%%%%%%%%%%%%%%%%%%%%%%%%%%%%%%%%%%%%%%%%%%%%%%%%%%%%%%%%%%%%%%%%%%%%%%%%%%%%%%%%%%%%%%%%%%%%%%%%%%%%%%%%%%%%%%%%%%%%%%%%%%%%%%%%%%%%%%%%%%%%%%%%%%%%%%%%%%%%%%%%%%%%%%%%%%%%%%%%%%%%%%%%%%%%%%%%%%%%%%%%%%%%%%%%%%%%%%%%%%%%%%%%%%%%%%%%%%%%%%%%%%%%%%%%%%%%%%%%%%%%%%%%%%%%%%%%%%%%%%%%%%%%%%%%%%%%%%%%%%%%%%%%%%%%%%%%%%%%%%%%%%%%%%%%%%%%%%%%%%%%%%%%%%%%%%%%%%%%%%%%%%%%%%%%%%%%%%%%%%%%%%%%%%%%%%%%%%%%%%%%%%%%%%%%%%%%%%%%%%%%%%%%%%%%%%%%%%%%%%%%%%%%%%%%%%%%%%%%%%%%%%%%%%%%%%%%%%%%%%%%%%%%%%%%%%%%%%%%%%%%%%%%%%%%%%%%%%%%%%%%%%%%%%%%%%%%%%%%%%%%%%%%%%%%%%%%%%%%%%%%%%%%%%%%%%%%%%%%%%%%%%%%%%%%%%%%%%%%%%%%%%%%%%%%%%%%%%%%%%%%%%%%%%%%%%%%%%%%%%%%%%%%%%%%%%%%%%%%%%%%%%%%%%%%%%%%%%%%%%%%%%%%%%%%%%%%%%%%%%%%%%%%%%%%%%%%%%%%%%%%%%%%%%%%%%%%%%%%%%%%%%%%%%%%%%%%%%%%%%%%%%%%%%%%%%%%%%%%%%%%%%%%%%%%%%%%%%%%%%%%%%%%%%%%%%%%%%%%%%%%%%%%%%%%%%%%%%%%%%%%%%%%%%%%%%%%%%%%%%%%%%%%%%%%%%%%%%%%%%%%%%%%%%%%%%%%%%%%%%%%%%%%%%%%%%%%%%%%%%

\section{A counterexample}

In this section we let $\bk$ be a field of characteristic 0 and choose $t\in\bk\backslash \mZ$. 

\subsection{Construction of the additive category}

The construction of additive monoidal categories, in particular rigid or symmetric ones, via generators and relations is ubiquitous in the literature. Here we employ the formalism and terminology from \cite{M}.

\subsubsection{} Fix the `type' $(1,0), (1,0)$. Then \cite[\S 3]{M} constructs the universal tensor category $\cC_{\mathrm{univ}}$ associated to it. It has generating object $W$ and morphisms $a,b:\unit\to W$. The following is a special case of \cite[Proposition~3.1]{M}.

\begin{lemma}\label{lem:M}
For any tensor category $\cB$, the assignment that sends a tensor functor $F:\cC_{\univ}\to\cB$ to the data $F(W)\xleftarrow{F(a)}\unit\xrightarrow{F(b)}F(W)$ produces a bijection between the set of isomorphism classes of tensor functors $\cC_{\univ}\to\cB$ and isomorphism classes of diagrams $\unit\rightrightarrows X$ in $\cB$ with rigid $X$.
\end{lemma}

\subsubsection{} We define the $\bk$-RSM category $\cC_t$ as the idempotent completion of the quotient of $\cC_{\univ}$ by the tensor ideal generated by the morphisms $\dim(W)-t\cdot \id_{\unit}$ and $a\otimes b$. We denote the images of $W,a,b$ in $\cC_t$ by $V_t,f,g$. The following corollary is thus immediately inherited from Lemma~\ref{lem:M}.

\begin{corollary}\label{cor:uni}
For a $\bk$-RSM category $\cB$, the assignment that sends a tensor functor $F:\cC_t\to\cB$ to the data $F(V_t)\xleftarrow{F(f)}\unit\xrightarrow{F(g)}F(V_t)$ produces a bijection between the set of isomorphism classes of tensor functors $\cC_t\to\cB$ and isomorphism classes of diagrams $x_1,x_2:\unit\rightrightarrows X$ in $\cB$ where $\dim(X)=t\cdot\id_{\unit}$ and $x_1\otimes x_2=0$.
\end{corollary}

\subsubsection{}\label{indecomposables3} We can repeat the above procedure but for the `empty type', in which case we recover the universal $\bk$-RSM categories from \cite[\S 10]{Del07}. Concretely, we will use the $\bk$-RSM categories $(\Rep \GL)_t$ and $(\Rep \GL)_{t-1}$ which are the universal $\bk$-RSM categories on objects $X_t,X_{t-1}$ of categorical dimensions $t\cdot\id_{\unit},(t-1)\cdot\id_{\unit}$. Since $t\not\in\mZ$, these are semisimple abelian categories, see \cite[Th\'eor\'eme~10.5]{Del07}, and the tensor unit $\unit$ is simple with $\bk=\End(\unit)$.

By \cite[Theorem~4.6.2]{CW}, the indecomposable (=simple) objects in $(\Rep \GL)_{t}$ are naturally labelled as $X_{\lambda,\mu}$, where $\lambda,\mu$ are arbitrary partitions. Here $X_{\lambda,\mu}$ is a direct summand of $X_t^{\otimes |\lambda|}\otimes (X_t^\ast)^{\otimes |\mu|}$, but not of $X_t^{\otimes m}\otimes (X_t^\ast)^{\otimes n}$ unless $(m,n)\ge (|\lambda|,|\mu|)$.

\begin{lemma}\label{lem:section}
\begin{enumerate}
\item The semisimplification of $\cC_t$ is $(\Rep \GL)_t$ with $\cC_t\to(\Rep \GL)_t$ given by $V_t\to X_t$ and $f,g\mapsto 0$.
\item The tensor functor $\Theta:(\Rep \GL)_t\to\cC_t$, $X_t\mapsto V_t$ is a tensor splitting of the semisimplification functor. 
\item The pseudo-abelian category $\cC_t$ is Krull-Schmidt. Every indecomposable object in~$\cC_t$ is isomorphic to $D_{\lambda,\mu}:=\Theta(X_{\lambda,\mu})$, for some partitions $\lambda,\mu$. Moreover, $\dim_{\bk}\End(D_{\lambda,\mu})=1$.
\end{enumerate}
\end{lemma}
\begin{proof}
We start by a description of morphism spaces. By construction of $\cC_{\mathrm{univ}}$ in \cite[\S 3]{M}, all objects are direct sums of objects $W^{\otimes m}\otimes (W^\ast)^{\otimes n}$. The morphism spaces are spanned by diagrams in terms of $a$ and $b$. We can put a non-negative grading on these spaces where the degree is the total number of occurrences of $a$ and $b$. We can then observe that the space 
\[\Hom(W^{\otimes m}\otimes (W^\ast)^{\otimes n},W^{\otimes r}\otimes (W^\ast)^{\otimes s})\] lives in degree $r-m+n-s$ and is of finite rank (since the degree places a bound on the number of morphisms $a$ and $b$ that can be used) over the polynomial algebra $\End_{\cC_{\univ}}(\unit)= \bk[\dim W]$. In particular, in the quotient of $\cC_{\univ}$ by $\dim W- t$, so certainly in $\cC_t$, the morphism spaces are finite-dimensional and $\End_{\cC_t}(\unit)=\bk$.

By $\End_{\cC_t}(\unit)=\bk$, the semisimplification of $\cC_t$ is well-defined. The tensor functor $\cC_t\to(\Rep \GL)_t$ as described in (1) via Corollary~\ref{cor:uni} is easily seen to be essentially surjective and full (since all morphisms in $(\Rep \GL)_t$ come from braiding and (co)evaluation). Since the target is semisimple, the conclusion in (1) follows.

Part (2) follows from the universal property of $(\Rep\GL)_t$, see \cite[Proposition~10.3]{Del07}.

As a consequence of the first paragraph, $\cC_t$ is Krull-Schmidt, so by construction, every indecomposable object must be a summand of a mixed tensor power $V_t^{\otimes m}\otimes (V_t^\ast)^{\otimes n}$.
Furthermore, by the degree consideration in the first paragraph, the braiding and evaluation $x\mapsto \dim W$ produce isomorphisms $\bk S_n\otimes \bk[x]\to\End(W^{\otimes n})$ for $\cC_{\univ}$, so that  $\bk S_n\to\End(V_t^{\otimes n})$ is an epimorphism. By part (1) and adjunction, it then follows that $\Theta$ produces isomorphisms
\[\End(X_t^{\otimes m}\otimes (X_t^\ast)^{\otimes n})\xrightarrow{\sim}\End(V_t^{\otimes m}\otimes (V_t^\ast)^{\otimes n}).\]
The decompositions into indecomposables in $\cC_t$, and their endomorphism algebras, are thus inherited from $(\Rep\GL)_t$.
\end{proof}

\subsubsection{} 
The classification of indecomposable objects in $(\Rep \GL)_{t-1}$ is identical to the one in $(\Rep \GL)_{t}$, so to avoid confusion we denote the corresponding objects by $Y_{\lambda,\mu}$.

We use Corollary~\ref{cor:uni} to define the tensor functor 
\begin{equation}\label{eq:Phi0}
\Phi_0:\cC_t\to (\Rep\GL)_{t-1},\quad V_t\mapsto X_{t-1}\oplus \unit, \quad f\mapsto (0,\id_{\unit}),\quad g\mapsto 0.
\end{equation}

\begin{lemma}\label{lem:lower}
Under the tensor functor $\Phi_0$,
the object $D_{\lambda,\mu}$ is sent to the direct sum of $Y_{\lambda,\mu}$ with simple objects of the form $Y_{\nu,\kappa}$ for $(|\nu|,|\kappa|)<(|\lambda|,|\mu|)$.

\end{lemma}
\begin{proof}
By definition of the objects $D_{\lambda,\mu}$, we can instead prove that
under the tensor functor
\[\Phi_0\circ\Theta:\;(\Rep \GL)_{t}\to(\Rep \GL)_{t-1}, \quad X_t\mapsto X_{t-1}\oplus\unit,\]
the object $X_{\lambda,\mu}$ is sent to the direct sum of $Y_{\lambda,\mu}$ with simple objects of the form $Y_{\nu,\kappa}$ for $(|\nu|,|\kappa|)<(|\lambda|,|\mu|)$. We start with the special case $\mu=\varnothing$. Now, using the fact that $X_{\omega,\varnothing}$ appears in $X_t^{\otimes |\omega|}$, and not in $X_t^{m}$ for $m<|\omega|$ (and the corresponding statements in $(\Rep \GL)_{t-1}$) and the fact that $X_{\lambda,\varnothing}$ is carved out by a primitive idempotent in
\[\End(X_t^{\otimes n})\cong\bk S_n\cong\End(X_{t-1}^{\otimes n}),\qquad\mbox{for }n=|\lambda|,\] it follows that $X_{\lambda,\varnothing}$ is sent to a direct sum of $Y_{\lambda,\varnothing}$ with $Y_{\omega,\varnothing}$ with $|\omega|<|\lambda|$. The case for $X_{\varnothing,\mu}$ follows identically. The general case then follows for instance from the fact that $X_{\lambda,\varnothing}\otimes X_{\varnothing,\mu}$ equals $X_{\lambda,\mu}$ up to lower summands, see \cite[Corollary~7.1.2]{CW}, and induction along $(|\lambda|,|\mu|)$.
\end{proof}

\subsection{The homotopy category}

From the $\bk$-RSM category $\cC_t$ we build the rigid tensor-triangulated category $\K^b(\cC_t)$, which will be our counterexample to the Nerves of Steel Conjecture. The tensor functor \eqref{eq:Phi0} then induces a tensor-triangulated (and homological) functor
\begin{equation}\label{eq:Phi}\Phi: \K^b(\cC_t)\to \K^b((\Rep \GL)_{t-1}).\end{equation}

\begin{proposition}\label{prop:00}
The only object sent to $0$ under $\Phi$ is $0$.
\end{proposition}
\begin{proof}
Let $A^\bullet$ be a (bounded) complex in $\cC_t$ representing an object in $\K^b(\cC_t)$ sent to zero by $\Phi$. If $A^\bullet$ is not zero will demonstrate that $A^\bullet\cong B^\bullet$ in $\K^b(\cC_t)$ for a complex $B^\bullet$ where $\oplus_i B^i$ has strictly fewer indecomposable direct summands in its expansion in the Krull-Schmidt category $\cC_t$ than $\oplus_iA^i$. Iteration of this observation proves the claim.

Consider a direct summand $D_{\lambda,\mu}$ of some $A^i$ such that $(|\lambda|,|\mu|)$ is maximal among all indecomposable direct summands of $\oplus_jA^j$. By Lemma~\ref{lem:lower}, $D_{\lambda,\mu}$ is sent to $Y_{\lambda,\mu}$ under $\Phi_0$ up to lower summands. Hence $\Phi(A^\bullet)=0$ implies that there must exist a direct summand $D_{\lambda,\mu}$ of $A^{i-1}$ or $A^{i+1}$ such that $A^{i-1}\to A^i$ or $A^i\to A^{i+1}$ composed with the corresponding inclusion and projection is non-zero and thus an isomorphism by Lemma~\ref{lem:section}(3). Without loss of generality, let that summand live in degree $i+1$, so that $A^\bullet$ looks like
\[ \xymatrix{&&& D_{\lambda,\mu}\ar@{=}[r]\ar[rdd]_(.2){\beta}&D_{\lambda,\mu}\\
\cdots\ar[r]^{d^{i-3}}&A^{i-2}\ar[r]^{d^{i-2}}& A^{i-1}\ar[rd]^{d^{i-1}}&\oplus &\oplus & A^{i+2}\ar[r]^{d^{i+2}}&\cdots\\
&&& R^{i}\ar[r]^{d^i}\ar[ruu]^(.2){\alpha}& R^{i+1}\ar[ru]^{d^{i+1}}}\]
for chosen complements $R^i,R^{i+1}$ in $A^{i},A^{i+1}$. Now $A^\bullet$ is isomorphic in $\K^b(\cC_t)$ to the complex
\[\cdots\xrightarrow{d^{i-2}}A^{i-1}\xrightarrow{d^{i-1}}R^i\xrightarrow{d^i-\beta\circ\alpha} R^{i+1}\xrightarrow{d^{i+1}} A^{i+2}\xrightarrow{d^{i+2}}\cdots,\]
which concludes the proof.
\end{proof}

\begin{corollary}\label{cor:local}
The tensor-triangulated category $\K^b(\cC_t)$ is local: the zero ideal is a prime thick tensor ideal.
\end{corollary}
\begin{proof}
Since $(\Rep\GL)_{t-1}$ is semisimple, its bounded homotopy category is equivalent to the (semisimple abelian) tensor category of $\mZ$-graded objects in $(\Rep\GL)_{t-1}$. Since no two non-zero objects in $(\Rep\GL)_{t-1}$ have zero tensor product (for non-zero $X\in (\Rep\GL)_{t-1}$, the object $X^\ast\otimes X$ contains the simple object $\unit$ as a direct summand and hence $X^\ast\otimes (X\otimes Y)$ contains $Y$ as a direct summand), the conclusion follows from Proposition~\ref{prop:00}.
\end{proof}

\subsection{The counterexample}

\begin{theorem}
The surjection \eqref{eq:twospec}
\[\Spc^{\mathrm{h}}(\K^b(\cC_t))\to \Spc(\K^b(\cC_t))\]
from \cite[Corollary~3.9]{Balmer20} is not a bijection.
\end{theorem}
\begin{proof}
By Corollary~\ref{cor:local}, the zero ideal is a prime thick tensor ideal in $\K^b(\cC_t)$.
We can prove that its fibre is not a singleton in two ways. We start with the easier, but perhaps less revealing proof. 

By \cite[Theorem~A.1]{Balmer20b} it is sufficient to demonstrate that $Z\otimes f^{\otimes n}\not=0$ for all $0\not=Z\in \K^b(\cC_t)$ and $n\in\mN$, as well as the corresponding claim for $g$. By symmetry, we only need to deal with $f$. It is sufficient to show that for non-zero $Z\in \K^b(\cC_t)$ and $n\in\mN$, the morphism
\[\Phi(Z\otimes f^{\otimes n})\;\cong\; \Phi(Z)\otimes \Phi(f)^{\otimes n}\]
is not zero, with $\Phi$ from \eqref{eq:Phi}. However, by \eqref{eq:Phi0} we have that $\Phi_0(f)$ is a split monomorphism $\unit\to \unit\oplus X_{t-1}$ and moreover $\Phi(Z)\not=0$ by Proposition~\ref{prop:00}, which concludes the proof.

Alternatively, we can consider the homological functors
\[\Phi,\Psi: \K^b(\cC_t)\rightrightarrows \cT:=\K^b((\Rep \GL)_{t-1}),\]
to the (semisimple rigid) abelian tensor category $\cT$, where $\Phi$ is as in \eqref{eq:Phi} and $\Psi$ is defined identically but with roles of $f$ and $g$ reversed. We have the universal homological functor (Yoneda embedding) $\K^b(\cC_t)\to\cA^{\fp}$ into the abelian category $\cA^{\fp}$ of finitely presented right $\K^b(\cC_t)$-modules, see \cite[Chapter 5]{NeBook}, \cite[Remark~2.6]{Balmer20} or \cite[Theorem~2.2.3 and Corollary~2.3.6(2)]{HomKer}. In particular we have exact tensor functors $\cA^{\fp}\rightrightarrows \cT$ defining Serre tensor ideals $K_{\Phi}$ and $K_{\Psi}$ in $\cA^{\fp}$ that intersect $\K^b(\cC_t)$ trivially by Proposition~\ref{prop:00}. By \cite[Lemma~3.8]{Balmer20} there is a maximal Serre tensor ideal in $\cA^{\fp}$ which contains $K_{\Phi}$ and which intersects $\K^b(\cC_t)$ trivially. The same also holds for $K_{\Psi}$ and, by definition of the homological spectrum, to conclude the proof it suffices for us to show that the two resulting maximal Serre tensor ideals cannot coincide. Since $\Phi$ sends $f$ to a monomorphism, we know that $K_\Phi$ contains $\ker f$, where we denote $f$ by the same symbol when viewed in $\cA^{\fp}$. If a proper Serre tensor ideal contained both $K_\Phi$ and $K_{\Psi}$, it would thus send both $f$ and $g$ to monomorphisms in the corresponding Serre quotient of $\cA^{\fp}$. But then $f\otimes g=0$ also has to be a monomorphism in this quotient, since tensor products with rigid objects are exact.  This can only be the case if $\unit=0$ in the quotient, contradicting that the ideal was proper.
\end{proof}

\subsection*{Acknowledgements} 
The author was partly supported by grants FT220100125 and DP250100762 from the Australian Research Council. The author thanks Logan Hyslop for pointing out \cite{BHR, Hyslop2}.

\end{document}